\documentclass{amsart}
\usepackage[T1]{fontenc}
\usepackage{amssymb, amsmath, amsthm}

\usepackage[backend=bibtex]{biblatex} %
\newtheorem{theorem}[subsection]{Theorem}

\newtheorem{corollary}[subsection]{Corollary}

\newtheorem{remark}[subsection]{Remark}

\newtheorem{example}[subsection]{Example}

\newtheorem{lemma}[subsection]{Lemma}

\newtheorem{question}[subsection]{Question}

\title{On continuous polynomials of the Macías space}

\author{Jhixon Macías}

\address{Jhixon Macías\newline
University of Puerto Rico at Mayaguez, Mayaguez, PR, USA\newline
United States of America}
\email{jhixon.macias@upr.edu}



\begin{document}

\begin{abstract}
Let $\mathbb{N}$ be the set of natural numbers. The Macías space $M(\mathbb{N})$ is the topological space $(\mathbb{N},\tau_M)$ where $\tau_M$ is generated by the collection of sets $\sigma_n := \{ m \in \mathbb{N} : \gcd(n, m) = 1 \}$.  In this paper, we characterize the continuity of polynomials over $ M(\mathbb{N})$ and prove that the only continuous polynomials are monomials.
\end{abstract}

\maketitle
\section{Introduction}\label{section1}

In 1953, M. Brown introduced a topology $\tau_G$ on the set of natural numbers $\mathbb{N}$, which is generated by the collection of arithmetic progressions $a+b\mathbb{N}$ where $a\in \mathbb{N}$ and $b\in\mathbb{N}\cup\{0\}$ such that $\gcd(a,b)=1$ (here the symbol $\gcd(a,b)$ means the greatest common divisor of $a$ and $b$). In \cite[Counterexample 60]{steen1978counterexamples} the topology $\tau_G$ is called the \textit{relatively prime integer topology}.  It was not until 1959 that the topology introduced by Brown was popularized by S. Golomb (now known as Golomb's topology) who in \cite{golomb1959connected} proves that Dirichlet's theorem on arithmetic progressions is equivalent to the set of prime numbers  being dense in the topological space $(\mathbb{N},\tau_G)$. 


Let $n \in \mathbb{N} \cup \{0\}$ and let $f:\mathbb{N} \to \mathbb{N}$ be a polynomial of the form  
\begin{equation}\label{eq2}  
f(x) = \sum_{k=0}^{n} a_k x^k \quad \text{where} \quad a_0, a_1, a_2, \dots, a_{n-1}, a_n \in \mathbb{N} \cup \{0\}.  
\end{equation}  

In \cite[Theorem 4.3]{szczuka2013darboux}, P. Szczuka proves that polynomials of the form (\ref{eq2}) with $a_n \in \mathbb{N}$ are continuous in $(\mathbb{N},\tau_G)$ if and only if $a_0 = 0$. Motivated by P. Szczuka's work, in this paper we study the continuity of polynomials of the form (\ref{eq2}) over the Macías topological space $M(\mathbb{N})$, a topology recently introduced and studied in \cite{jhixon2024,MACIAS2024109070}  which we will define in the next section. We prove that the only continuous polynomials in $M(\mathbb{N})$ are the constant polynomials and those of the form $ax^n$ with $a, n \in \mathbb{N}$ (see Theorem \ref{thcontinuo}).  

\section{The Macías space}\label{section2}

The Macías space $M(\mathbb{N})$ is the topological space $(\mathbb{N},\tau_M)$ where $\tau_M$ is generated by the collection of sets $\sigma_n := \{ m \in \mathbb{N} : \gcd(n, m) = 1 \}$. Note that \begin{equation*}
    \sigma_n=\bigcup_{\substack{1\leq m\leq n\\ \gcd(m,n)=1}}n\left(\mathbb{N}\cup \{0\}\right)+m,
\end{equation*}
since for every integer $x$ we have that $\gcd(n,m)=\gcd(n,nx+m)$; see \cite[Theorem 1.9]{niven1991introduction}. It is easily deduced from here that $\sigma_n$ is infinite for every positive integer $n$ and that the topology $\tau_M$ is strictly coarser than Golomb topology. The Macías space does not satisfy the $\mathrm{T}_0$ separation axiom, satisfies (vacuously) the $\mathrm{T}_4$ separation axiom, is  connected, locally connected, path-connected, and pseudocompact.   On the other hand, for $n,m\in\mathbb{N}$, it holds that $\sigma_{nm}=\sigma_{n}\cap\sigma_m$ and 
$\textbf{cl}_{M(\mathbb{N})}(\{nm\})=\textbf{cl}_{M(\mathbb{N})}(\{n\})\cap \textbf{cl}_{M(\mathbb{N})}(\{m\})$ where $\textbf{cl}_{M(\mathbb{N})}(\{n\})$ denotes the closure of the singleton set $\{n\}$ in the topological space $M(\mathbb{N})$ (see \cite{MACIAS2024109070}). Furthermore, in \cite{jhixon2024} (using the topology $\tau_M$), a topological proof of the infinitude of prime numbers is presented (different from the proofs of Furstenberg \cite{furstenberg1955infinitude} and Golomb \cite{golomb1959connected}). In the same work, the infinitude of any non-empty subset of prime numbers is characterized in the following sense: 

\begin{lemma}[\cite{jhixon2024}]\label{mainlema}
If $A$ is a non-empty subset of prime numbers, then $A$ is infinite if and only if $A$ is dense in $M(\mathbb{N})$. 
\end{lemma}

\begin{proof}
Suppose that  $A$ is infinite. Then, for any positive integer $n>1$, we can choose a prime $p\in A$ such that $p>n$, and consequently, $p\in\sigma_n$ since $\gcd(n,p)=1$. Therefore, $A$ is dense in $M(\mathbb{N})$. On the other hand, assume that $A$ is dense in $M(\mathbb{N})$. Let $\{p_1,p_2,\dots, p_k\}\subset A$ be a finite collection of prime numbers and consider the  non-empty basic element $\sigma_x$ where $x=p_1\cdot p_2\cdots p_k$. Note that none of the $p_i$ belong to $\sigma_x$, but since $A$ is dense in $M(\mathbb{N})$, there must be another prime number $q$, different from each $p_i$, such that $q\in\sigma_x$. Consequently, $A$ is infinite. 
\end{proof}




The Lemma \ref{mainlema}, together with the following theorems, will be very useful in Section \ref{section3}. 

\begin{theorem}[\cite{dirichlet1837beweis}]\label{thdirichlet}
Let $n,m\in\mathbb{N}$ such that $\gcd(n,m)=1$. Then,   the set $n(\mathbb{N}\cup\{0\})+m$  contains an infinite number of prime numbers.
\end{theorem}

\begin{proof}
The reader can check an \textit{elementary proof} in \cite[Chapter 7]{apostol2013introduction}.
\end{proof}

\begin{theorem}[\cite{schur1912uber}]\label{schurtheorem}
Let $f(x)$ be a non-constant polynomial of the form (\ref{eq2}). Then the set of prime numbers dividing some element of the set $\{f(n)\neq 0:n\in\mathbb{N}\}$ is infinite. 
\end{theorem}
\begin{proof}
    The reader may consult a proof in \cite{murty2006primes}.
\end{proof}

\section{On continuous polynomials of the Macías space}\label{section3}

Naturally, the following question arises: Does \cite[Theorem 4.3]{szczuka2013darboux} hold in the space $M(\mathbb{N})$? The answer is no. Consider the following example.  

\begin{example}\label{ejem1}  
Consider the polynomial of the form (\ref{eq2}) given by $f(x) = x^2 + x = x(x+1)$. Note that $f(x) \equiv 0 \pmod{7}$ has solutions in the complete residue system modulo $7$, namely $x = 6,7$. Thus,  
\begin{equation*}  
W := f^{-1}(\sigma_7) = \bigcup_{i=1}^{5} 7(\mathbb{N} \cup \{0\}) + i.  
\end{equation*}  
Now, by Dirichlet's theorem on arithmetic progressions (Theorem \ref{thdirichlet}) and Lemma \ref{mainlema}, the set $W' := 7(\mathbb{N} \cup \{0\}) + 6$, which is disjoint from $W$, is dense in $M(\mathbb{N})$. Thus, if $W$ were open, then $W \cap W' \neq \emptyset$, which is a contradiction. Therefore, $f$ is not continuous in $M(\mathbb{N})$.  
\end{example}  

In general, continuity is also not guaranteed for polynomials of the form (\ref{eq2}) when $a_0 > 0$.  

\begin{example}\label{ejem2}  
Consider the polynomial of the form (\ref{eq2}) given by $f(x) = x^2 + 4x + 2$. Note that $f(x) \equiv 0 \pmod{7}$ has solutions in the complete residue system modulo $7$, namely $x = 1,2$. Thus,  
\begin{equation*}  
W := f^{-1}(\sigma_7) = \bigcup_{i=3}^{7} 7(\mathbb{N} \cup \{0\}) + i.  
\end{equation*}  
Now, by Dirichlet's theorem on arithmetic progressions (Theorem \ref{thdirichlet}) and Lemma \ref{mainlema}, the set  
\[ W' := 7(\mathbb{N} \cup \{0\}) + 1 \cup 7(\mathbb{N} \cup \{0\}) + 2, \]  
which is disjoint from $W$, is dense in $M(\mathbb{N})$. Thus, if $W$ were open, then $W \cap W' \neq \emptyset$, which is a contradiction. Therefore, $f$ is not continuous in $M(\mathbb{N})$.  
\end{example}  

The following theorem characterizes the continuity of polynomials of the form (\ref{eq2}) in $M(\mathbb{N})$.

\begin{theorem}\label{mainth}  
Let $f$ be a polynomial of the form (\ref{eq2}). Then, $f$ is continuous if and only if $f^{-1}(\sigma_p) = \sigma_p$, $f^{-1}(\sigma_p) = \mathbb{N}$, or $f^{-1}(\sigma_p) = \emptyset$ for every prime number $p$.  
\end{theorem}  

\begin{proof}  
Suppose that $f^{-1}(\sigma_p) = \sigma_p$, $f^{-1}(\sigma_p) = \mathbb{N}$, or $f^{-1}(\sigma_p) = \emptyset$ for every prime number $p$. Consider a basic element $\sigma_k$. Clearly, if $k = 1$, then $f^{-1}(\sigma_k) = \mathbb{N}$. If $k > 1$, consider its prime factorization, say $k = p_1^{\alpha_1} p_2^{\alpha_2} \cdots p_t^{\alpha_t}$. Then,  
\begin{equation*}  
    f^{-1}(\sigma_k) = f^{-1} \left( \bigcap_{i=1}^{t} \sigma_{p_i} \right) = \bigcap_{i=1}^{t} f^{-1}(\sigma_{p_i}).  
\end{equation*}  
By assumption, $f^{-1}(\sigma_{p_i}) = \sigma_{p_i}$, $f^{-1}(\sigma_{p_i}) = \mathbb{N}$, or $f^{-1}(\sigma_{p_i}) = \emptyset$ for each $i \in \{1,2, \dots, t\}$, all of which are open sets. Thus, $f$ is continuous in $M(\mathbb{N})$.  

Conversely, suppose that $f$ is continuous in $M(\mathbb{N})$. Then,  
\begin{equation*}  
    f^{-1}(\sigma_p) = \{x \in \mathbb{N} : \gcd(f(x), p) = 1\} = \{x \in \mathbb{N} : f(x) \not\equiv 0 \pmod{p} \},  
\end{equation*}  
is an open set in $M(\mathbb{N})$. Now, suppose that for some prime number $p$, the set $f^{-1}(\sigma_p)$ is non-empty, different from $\sigma_p$, and different from $\mathbb{N}$.  Note that if $f(x)\not\equiv 0 \pmod{p}$ for all $x\in \{1,2,\dots, p-1\}$, then $\sigma_p\subsetneq f^{-1}(\sigma_p)$, which implies that $f^{-1}(\sigma_p)\cap p\mathbb{N}\neq\emptyset$. But this last statement implies that $p\mathbb{N}\subset f^{-1}(\sigma_p)$. Therefore, $\mathbb{N}=\sigma_p\cup p\mathbb{N}\subset f^{-1}(\sigma_p)\subset \mathbb{N}$, which is absurd because $f^{-1}(\sigma_p)\neq\mathbb{N}$.  Hence, there exists at least one $x \in \{1,2, \dots, p-1\}$ such that $f(x) \equiv 0 \pmod{p}$. Then, 
\begin{equation*}  
    f^{-1}(\sigma_p) \cap p(\mathbb{N} \cup \{0\}) + x = \emptyset \ \ \text{for some} \ \ x \in \{1,2, \dots, p-1\}.
\end{equation*}  
However, by Dirichlet's theorem on arithmetic progressions (Theorem \ref{thdirichlet}) and Lemma \ref{mainlema}, the set $p(\mathbb{N} \cup \{0\}) + x$ is dense in $M(\mathbb{N})$, and so $f^{-1}(\sigma_p) \cap p(\mathbb{N} \cup \{0\}) + x \neq \emptyset$ (a contradiction). Thus, $f^{-1}(\sigma_p) = \sigma_p$, $f^{-1}(\sigma_p) = \mathbb{N}$, or $f^{-1}(\sigma_p) = \emptyset$ for every prime number $p$.  
\end{proof}  

An immediate application of Theorem \ref{mainth} is to identify families of continuous and non-continuous polynomials in $M(\mathbb{N})$.


\begin{lemma}\label{lem4}  
Polynomials \( f:\mathbb{N}\to\mathbb{N} \) of the form \( f(x)=a \), where \( a \) is a fixed positive integer (constant polynomials), are continuous in \( M(\mathbb{N}) \).  
\end{lemma}  

\begin{proof}  
Consider a basic element \( \sigma_p \) with \( p \) a prime number. Note that if \( \gcd(a,p) = 1 \), then \( f^{-1}(\sigma_p) = \mathbb{N} \). Otherwise (if \( \gcd(a,p) > 1 \)), we have \( f^{-1}(\sigma_p) = \emptyset \). In either case, by Theorem \ref{mainth}, it follows that \( f \) is continuous in \( M(\mathbb{N}) \).  
\end{proof}

\begin{lemma}\label{lem5}  
The polynomials $f:\mathbb{N} \to \mathbb{N}$ of the form $f(x) = x^n$ are continuous in $M(\mathbb{N})$.  
\end{lemma}  

\begin{proof}  
Let $p$ be a prime number. Then,  
\begin{equation*}  
    f^{-1}(\sigma_p) = \{x \in \mathbb{N} : \gcd(x^n, p) = 1\} = \{x \in \mathbb{N} : \gcd(x, p) = 1\} = \sigma_p.  
\end{equation*}  
Thus, by Theorem \ref{mainth}, $f$ is continuous in $M(\mathbb{N})$.  
\end{proof}  

\begin{lemma}\label{lemextra}  
Let $f, g : \mathbb{N} \to \mathbb{N}$ be polynomials of the form (\ref{eq2}), either continuous or constant. Then, $h(x) := f(x) g(x)$ is continuous in $M(\mathbb{N})$.  
\end{lemma}  

\begin{proof}  
Let $p$ be a prime number. Note that $h^{-1}(\sigma_p)=f^{-1}(\sigma_p)\cap h^{-1}(\sigma_p)$. If $f$ and $g$ are constant, then  
\begin{equation*}  
    f^{-1}(\sigma_p), g^{-1}(\sigma_p) \in \{\emptyset, \mathbb{N} \}.  
\end{equation*}  
In this case, $h^{-1}(\sigma_p) \in \{\emptyset, \mathbb{N} \}$.  

On the other hand, suppose without loss of generality that $f$ is of the form (\ref{eq2}) and $g$ is constant. Then, by Theorem \ref{mainth},  
\begin{equation*}  
    f^{-1}(\sigma_p) \in \{\emptyset, \sigma_p, \mathbb{N} \} \quad \text{and} \quad g^{-1}(\sigma_p) \in \{\emptyset, \mathbb{N} \}.  
\end{equation*}  
In this case, $h^{-1}(\sigma_p) \in \{\emptyset, \sigma_p, \mathbb{N} \}$.  

Finally, if both $f$ and $g$ are of the form (\ref{eq2}) and continuous, then  
\begin{equation*}  
    f^{-1}(\sigma_p), g^{-1}(\sigma_p) \in \{\emptyset, \sigma_p, \mathbb{N} \}.  
\end{equation*}  
Thus, $h^{-1}(\sigma_p) \in \{\emptyset, \sigma_p, \mathbb{N} \}$. In any case, by Theorem \ref{mainth}, $h$ is continuous in $M(\mathbb{N})$.  
\end{proof}  


\begin{remark}
In general, since $M(\mathbb{N})$ is a topological semigroup (see \cite[Theorem 2.3]{MACIAS2024109070}), the product of two continuous functions on $M(\mathbb{N})$ results in another continuous function on $M(\mathbb{N})$.
\end{remark}

\begin{corollary}\label{cor2}  
The polynomials $f:\mathbb{N} \to \mathbb{N}$ of the form $f(x) = a x^n$ with $a \in \mathbb{N}$ are continuous in $M(\mathbb{N})$.  
\end{corollary}  

\begin{proof}  
This follows from Lemma \ref{lemextra}.  
\end{proof}

The examples \ref{ejem1} and \ref{ejem2}, along with the results presented so far, motivate the following question:  

\begin{question}\label{ques1}  
Are the only continuous polynomials in $M(\mathbb{N})$ the constant polynomials and those of the form $ax^n$ with $a, n \in \mathbb{N}$?  
\end{question}  

The answer to Question \ref{ques1} is affirmative; see Theorem \ref{thcontinuo}. First, consider the following pair of lemmas.  

\begin{lemma}\label{lem7}  
Let $k \in \mathbb{N}$. Suppose $f:\mathbb{N} \to \mathbb{N}$ is a polynomial of the form (\ref{eq2}) that is discontinuous in $M(\mathbb{N})$. Then, $h(x):=x^k f(x)$ is also discontinuous in $M(\mathbb{N})$.  
\end{lemma}  

\begin{proof}  
Since $f$ is discontinuous in $M(\mathbb{N})$, by Theorem \ref{mainth}, there exists a prime number $p$ such that $\emptyset \neq f^{-1}(\sigma_p) \cap \sigma_p \subsetneq \sigma_p$. Then,  
\begin{equation*} 
 h^{-1}(\sigma_p)= \{x\in\mathbb{N}:\gcd(x^k,p)=1\}\cap f^{-1}(\sigma_p)=\sigma_p \cap f^{-1}(\sigma_p) \subsetneq \sigma_p, 
\end{equation*}  
which cannot be an open set by the same argument as in Theorem \ref{mainth}. Therefore, $h$ cannot be continuous in $M(\mathbb{N})$.  
\end{proof}  

\begin{lemma}\label{lem8}  
Let $f(x)$ be a polynomial of the form (\ref{eq2}) with $a_0 \neq 0$ and at least one $a_i \neq 0$ for some $i \in \{1,2,3,\dots, n\}$. Then, $f(x)$ is not continuous in $M(\mathbb{N})$.  
\end{lemma}  

\begin{proof}  
It suffices to show that there exists at least one prime number $p > \sum_{i=0}^{n} a_i$ such that $f(x) \equiv 0 \pmod{p}$ for some $x \in \{1,2,3,\dots,p-1,p\}$. In this case, necessarily $1 < x < p$ since $x=p$ can not be a solution of $f(x)\equiv 0 \pmod p$ because $\gcd(a_0,p)=1$. Then, by Theorem \ref{mainth}, $f(x)$ is not continuous in $M(\mathbb{N})$. 

Suppose that for every prime number $p > \sum_{i=0}^{n} a_i$, we have $f(x) \not\equiv 0 \pmod{p}$ for all $x \in \{1,2,\dots,p-1,p\}$.  Then, the set of prime numbers $p$ that divide some element of the set $\{f(n)\neq 0 : n \in \mathbb{N}\}$ is finite, which is a contradiction by Schur's theorem (Theorem \ref{schurtheorem}).
\end{proof}  



\begin{theorem}\label{thcontinuo}  
The only continuous polynomials in $M(\mathbb{N})$ are the polynomials the form $ax^n$ with $a\in \mathbb{N}$ and $n\in\mathbb{N}\cup\{0\}$.  
\end{theorem}  

\begin{proof}  
By Lemma \ref{lem4}, Corollary \ref{cor2}, and Lemma \ref{lem8}, it suffices to show that polynomials $f(x)$ of the form (\ref{eq2}) with $a_0 = 0$ and $a_i, a_j \neq 0$ for some distinct pair $i, j \in \{1,2,3,\dots, n\}$ are discontinuous in $M(\mathbb{N})$. 

Let $f(x)$ be a polynomial of the previously described form. In this case, we can write $f(x) = x g(x)$, where $g(x)$ is a polynomial of the form described in Lemma \ref{lem8}. Thus, by Lemma \ref{lem7}, we conclude that $f(x)$ is not continuous in $M(\mathbb{N})$.  
\end{proof}

\section{Final comment}  

The statement of Theorem \ref{mainth} can be extended to functions $f:\mathbb{N}\to\mathbb{N}$ such that if $x\equiv y \mod n$ then $f(x)\equiv f(y) \mod n$ for $n,x,y\in \mathbb{N}$. 



\begin{theorem}\label{thfinal}
Let $f:\mathbb{N}\to\mathbb{N}$ be a function such that if $x\equiv y \mod n$ then $f(x)\equiv f(y) \mod n$ for $n,x,y\in \mathbb{N}$. Then $f$ is continuous in $M(\mathbb{N})$ if and only if $f^{-1}(\sigma_p) = \sigma_p$, $f^{-1}(\sigma_p) = \mathbb{N}$, or $f^{-1}(\sigma_p) = \emptyset$ for every prime number $p$. 
\end{theorem}

\begin{proof}
Similar to the proof of Theorem \ref{mainth}. Note that the fact that $f$ is a function such that if $x\equiv y \mod n$ then $f(x)\equiv f(y) \mod n$ for $n,x,y\in \mathbb{N}$ allows us to repeat the argument of the converse of the proof of Theorem \ref{mainth}. 
\end{proof}

For example, functions $f:\mathbb{N}\to\mathbb{N}$ of the form $f(x)=a^x+b$ with $a,b\in\mathbb{N}$ and $a\neq b$ are not continuous in $M(\mathbb{N})$ by applying the theory of \textit{primitive roots} $\pmod p$ with $p$ a prime number and Theorem \ref{thfinal}. Of course, if we take a sufficiently large prime number $p$ ($>a,b$) such that $b\not\equiv 1 \pmod p$, then $f(x)\equiv 0 \pmod p$ would have a unique nontrivial solution ($x\not\equiv0 \pmod p)$ in the complete residue system $\pmod p$, so that $f^{-1}(\sigma_p)\notin\{\emptyset,\sigma_p,\mathbb{N}\}$ and therefore $f$ would not be continuous in $M(\mathbb{N})$. However, functions $f:\mathbb{N}\to \mathbb{N}$ of the form $a^x$ with $a\in\mathbb{N}$ would be continuous in $M(\mathbb{N})$, indeed, in this case, it is easy to verify that for any prime number $p$, we would have $f^{-1}(\sigma_p)\in\{\emptyset,\mathbb{N}\}$. In summary, the only \textit{exponential functions} that are continuous in $M(\mathbb{N})$ are those of the form $a^x$ with $a\in\mathbb{N}$.  

With the previously mentioned discussion, we aim to motivate the study of continuous functions in general over $M(\mathbb{N})$, perhaps conducting a study similar (as far as possible) to the one carried out by T. Banakh, J. Mioduszewski, and S. Turek in \cite{banakh2017continuous} on continuous functions in general over the Golomb topological space.

It is interesting that in the Golomb topological space, there are \textit{more} continuous polynomials than in the Macías topological space. However, in the Macías topological space, there are \textit{more} homeomorphisms than in the Golomb topological space; see \cite{banakh2019golomb, macias2024self}.

\printbibliography
\end{document}